\documentclass[reqno]{amsart}
\usepackage{mathrsfs}
\usepackage{color}
\usepackage{amsmath}
\usepackage{amsfonts}
\usepackage{amssymb}
\usepackage{graphicx}
\usepackage{hyperref}

 \newtheorem{Theorem}{Theorem}[section]
 \newtheorem{Corollary}[Theorem]{Corollary}
 \newtheorem{Lemma}[Theorem]{Lemma}
 \newtheorem{Proposition}[Theorem]{Proposition}
 
 \newtheorem{Definition}[Theorem]{Definition}

 \newtheorem{Conjecture}[Theorem]{Conjecture}

 \newtheorem{Remark}[Theorem]{Remark}

 \numberwithin{equation}{section}
\begin{document} 

\title[A generalization of the conjugate Hardy $H^2$ spaces]
 {A generalization of the conjugate Hardy $H^2$ spaces}

\author{Qi'an Guan}
\address{Qi'an Guan: School of
Mathematical Sciences, Peking University, Beijing 100871, China.}
\email{guanqian@math.pku.edu.cn}

\author{Zheng Yuan}
\address{Zheng Yuan: Institute of Mathematics, Academy of Mathematics and Systems
Science, Chinese Academy of Sciences, Beijing 100190, China.}
\email{yuanzheng@amss.ac.cn}

\thanks{}



\date{}

\dedicatory{}

\commby{}

%%% ----------------------------------------------------------------------

\begin{abstract}
In this article, we consider a generalization of the conjugate Hardy $H^2$ spaces, and give some properties of the minimal norm of the generalization and some relations between the norm of the generalization and the minimal $L^2$ integrals. As applications, we give some monotonicity results for the conjugate Hardy $H^2$ kernels and the Bergman kernels on planar regions, and some relations between the conjugate Hardy $H^2$ kernels and the Bergman kernels on planar regions.
\end{abstract}

%%% ----------------------------------------------------------------------
\maketitle
%%% ----------------------------------------------------------------------

\section{Introduction}

Let $D$ be a planar region bounded by finite analytic Jordan curves.  
Let $z_0\in D$, and $G_{D}(\cdot,z_0)$ be the Green function on $D$. The conjugate Hardy $H^2$ space $H^{2}_{(c)}(D)$ (see \cite{saitoh}) is defined as the set of all holomorphic functions $f(z)$ on $D$ such that the subharmonic functions $|f(z)|^2$ have harmonic majorants $U(z)$: 
$$|f(z)|^2\le U(z)\,\, \text{on}\,\,D.$$
Each function $f(z)\in H^{2}_{(c)}(D)$ has Fatou's nontangential boundary value a.e. on $\partial D$ belonging to $L^2(\partial D)$ (see \cite{duren}). We recall the norm on $H^{2}_{(c)}(D)$ (see \cite{saitoh}):
$$\|f\|_{H^{2}_{(c)}(D)}:=\left(\frac{1}{2\pi}\int_{\partial D}|f(z)|^2\left(\frac{\partial G_D(z,z_0)}{\partial v_z}\right)^{-1}|dz|\right)^{\frac{1}{2}}$$
for any $f\in H^{2}_{(c)}(D)$, where  $\partial/\partial v_z$ denotes the derivative along the outer normal unit vector $v_z$. 

The conjugate Hardy $H^2$ kernel (see \cite{saitoh}) was defined by
\begin{equation}
	\nonumber
	\hat K_{D}(z_0):=\frac{1}{\inf\left\{\|f\|^2_{H^2_{(c)}(D)}:f\in H^2_{(c)}(D)\,\&\,f(z_0)=1\right\}}.\end{equation}
	Saitoh posed the following conjecture for the conjugate Hardy $H^2$ kernel (see \cite{saitoh,yamada}):
	
\begin{Conjecture}\label{conj1}
		If $D$ is not simple connected, then $\hat K_{D}(z_0)>\pi B_D(z_0)$, where $B_D(z_0)$ is the Bergman kernel on $D$.
\end{Conjecture}
In \cite{guan-19saitoh}, Guan gave an affirmative answer to the above conjecture  by using a concavity  property for minimal $L^2$ integrals (see \cite{ber-lem,guan_sharp}) and a solution of the equality part of Suita conjecture (see \cite{guan-zhou13ap}). Some recent results about general versions  of Saitoh's conjecture can be referred to \cite{GY-weightedsaitoh,GY-saitohprodct,GY-saitohfinite}.

In \cite{Ohsawa6}, Ohsawa gave an extension theorem involving the Hardy $H^2$ space as a limiting case for the weighted Bergman spaces.

Inspired by Saitoh's work in \cite{saitoh} and Ohsawa's work in \cite{Ohsawa6}, we consider a generalization of the conjugate Hardy $H^2$ spaces on complex manifolds in this article.

Let $M$ be an $n-$dimensional complex manifold, and let $K_{M}$ be the canonical line bundle on $M$. Let $\psi<0$ be a plurisubharmonic function on $M$, and let $\rho\ge0$ be a Lebesgue measurable function on $M$.

\begin{Definition}	The generalization of the conjugate Hardy $H^2$ space  is defined as the set of all $f\in H^0(M,\mathcal{O}(K_M))$ such that
	$$\liminf_{t\rightarrow0+0}\frac{1}{t}\int_{\{-t<\psi<0\}}|f|^2\rho<+\infty,$$
	which is called Ohsawa-Saitoh-Hardy space and denoted by  $OSH^2_{\rho}(M,\psi)$.
	 We define a norm on the space $OSH^2_{\rho}(M,\psi)$ as follows: 
	$$\|f\|_{OSH^2_{\rho}(M,\psi)}:=\left(\liminf_{t\rightarrow0+0}\frac{1}{t}\int_{\{-t<\psi<0\}}|f|^2\rho\right)^{\frac{1}{2}}$$ 
	for any $f\in OSH^2_{\rho}(M,\psi)$, where $|f|^2:=\sqrt{-1}^{n^2}f\wedge\bar f$. 
	\end{Definition}

	If $M=D$ is a planar region bounded by finite analytic Jordan curves and $\rho\equiv1$, the Ohsawa-Saitoh-Hardy space $OSH^2(M,\psi)$ is $\{fdz:f\in H^2_{(c)}(D)\}$ and $\|fdz\|^2_{OSH^2(M,\psi)}=2\pi\|f\|^2_{H^2_{(c)}(D)}$ for any $f\in H^2_{(c)}(D)$  (for the proof, see  Proposition \ref{l:OSH=H}). Thus, the Ohsawa-Saitoh-Hardy space $OSH^2_{\rho}(M,\psi)$ is a generalization of the conjugate Hardy $H^2$ spaces.

 Let $X$ and $Z$ be closed subsets of $M$, and assume that $(M,X,Z)$  satisfies the following conditions:

$\uppercase\expandafter{\romannumeral1}.$ $X$ is a closed subset of $M$ and $X$ is locally negligible with respect to $L^2$ holomorphic functions, i.e., for any local coordinated neighborhood $U\subset M$ and for any $L^2$ holomorphic function $f$ on $U\backslash X$, there exists an $L^2$ holomorphic function $\tilde{f}$ on $U$ such that $\tilde{f}|_{U\backslash X}=f$ with the same $L^2$ norm;

$\uppercase\expandafter{\romannumeral2}.$ $Z$ is an analytic subset of $M$ and $M\backslash (X\cup Z)$ is a weakly pseudoconvex K\"ahler manifold.

\

 Let $\varphi$ be a Lebesgue measurable function on $M$, such that $\psi+\varphi$ is a plurisubharmonic function on $M$. 
\begin{Definition}
We call a positive measurable function $c$ on $(0,+\infty)$ in class $P_{0,M}$ if the following two statements hold:
\par
$(1)$ $c(t)e^{-t}$ is decreasing with respect to $t$;
\par
$(2)$ there is a closed subset $E$ of $M$ such that $E\subset Z\cap \{\psi(z)=-\infty\}$ and for any compact subset $K\subset M\backslash E$, $e^{-\varphi}c(-\psi)$ has a positive lower bound on $K$.
\end{Definition}

 Let $Z_0$ be a subset of $\{\psi=-\infty\}$ such that $Z_0 \cap
Supp(\mathcal{O}/\mathcal{I}(\varphi+\psi))\neq \emptyset$. Let $U \supset Z_0$ be
an open subset of $M$, and let $f$ be a holomorphic $(n,0)$ form on $U$. Let $\mathcal{F}_{z_0} \supset \mathcal{I}(\varphi+\psi)_{z_0}$ be an ideal of $\mathcal{O}_{z_0}$ for any $z_{0}\in Z_0$. Let $c\in P_{0,M}$ such that $\int_0^{+\infty}c(t)e^{-t}dt<+\infty$.

Denote
\begin{equation}\nonumber
\begin{split}
\inf\Bigg\{\int_{ \{ \psi<-t\}}|\tilde{f}|^2e^{-\varphi}c(-\psi): \tilde{f}\in
H^0(\{\psi<-t\},\mathcal{O}(K_M)) \\
\&\, (\tilde{f}-f)\in
H^0(Z_0 ,(\mathcal{O} (K_M) \otimes \mathcal{F})|_{Z_0})\Bigg\}
\end{split}
\end{equation}
by $G(t)$, where $t\in[0,+\infty)$  and $(\tilde{f}-f)\in
H^0(Z_0 ,(\mathcal{O} (K_M) \otimes \mathcal{F})|_{Z_0})$ means $(\tilde{f}-f,z_0)\in(\mathcal{O}(K_M)\otimes \mathcal{F})_{z_0}$ for all $z_0\in Z_0$.

 Assume that there exists $t \in [0,+\infty)$ satisfying that $G(t)<+\infty$. Theorem \ref{thm:concave} shows that $G(h^{-1}(r))$ is concave on $[0,\int_0^{+\infty}c(t)e^{-t}]$, then the right-hand derivative $G'_{+}(t):=\lim_{t_1\rightarrow0+0}\frac{G(t+t_1)-G(t)}{t_1}$ exists and is finite for any $t\ge0$.

Let $\rho=e^{-\varphi}c(-\psi)$ on $M$. Denote that $M_t:=\{\psi<-t\}$ and $\psi_t=\psi+t$ for any $t\ge0$. Let $F_t$ be the unique holomorphic $(n,0)$ form on $M_t$ (see Lemma \ref{lem:A}) satisfying that $(F_t-f)\in
H^0(Z_0 ,(\mathcal{O} (K_M) \otimes \mathcal{F})|_{Z_0})$ and $G(t)=\int_{ \{ \psi<-t\}}|F_t|^2e^{-\varphi}c(-\psi)$. By definitions, we have $F_t\in OSH_{\rho}^2(M_t,\psi_t)$ and 
\begin{equation}
	\label{eq:0522a}\|F_t\|_{OSH_{\rho}^2(M_t,\psi_t)}^2=\liminf_{B\rightarrow0+0}\frac{\int_{\{-t-B<\psi<-t\}}|F_t|^2e^{-\varphi}c(-\psi)}{B}\leq-G'_{+}(t).
\end{equation}

	Denote that $c_+(t)=\lim_{s\rightarrow t+0}c(s)$. The concavity of $G(h^{-1}(r))$ shows that $\frac{e^{t}}{c_+(t)}G'_+(t)$ is decreasing on $[0,+\infty)$. We give some properties for $\|F_t\|_{OSH_{\rho}^2(M_t,\psi_t)}^2$ as follows:  
	
	\begin{Theorem}
	\label{thm1}
	The following three statements hold:
	
	$(1)$
	 $\frac{e^{t}}{c_+(t)}\|F_t\|_{OSH_{\rho}^2(M_t,\psi_t)}^2$ is increasing on $[0,+\infty)$; 
	 
	 $(2)$ $\|F_t\|_{OSH_{\rho}^2(M_t,\psi_t)}^2=-G'_{+}(t)$
	 holds for a.e. $t\in(0,+\infty)$;
	 
	 $(3)$ when functions $G'_+$ and $c$ are both continuous at $t\in(0,+\infty)$, equality 
	$$\|F_t\|_{OSH_{\rho}^2(M_t,\psi_t)}^2=-G'_{+}(t)$$
	 holds. 
	\end{Theorem}
	
Denote that 
	\begin{equation}\nonumber
\begin{split}
H(t):=\inf\Bigg\{\|\tilde f\|_{OSH_{\rho}^2(M_t,\psi_t)}^2: \tilde{f}\in
H^0(\{\psi<-t\},\mathcal{O}(K_M)) \\
\&\, (\tilde{f}-f)\in
H^0(Z_0 ,(\mathcal{O} (K_M) \otimes \mathcal{F})|_{Z_0})\Bigg\}
\end{split}
\end{equation}	
for any $t\ge0$.

Following from inequality \eqref{eq:0522a}, we have
\begin{equation}
	\label{eq:0718a}
	H(t)\le -G'_{+}(t)
\end{equation}
 for any $t\ge0$. We gives a necessary and sufficient condition
 for $H(t)=-G'_{+}(t)$ for a.e. $t\in(0,+\infty)$.

\begin{Theorem}\label{thm2}
	The following two statements are equivalent:
	
	$(1)$ $H(t)=-G'_{+}(t)$ holds for a.e. $t\in(0,+\infty)$;
	
	$(2)$ there exists $F\in H^0(M,\mathcal{O}(K_M))$ such that $(F-f)\in
H^0(Z_0 ,(\mathcal{O} (K_M) \otimes \mathcal{F})|_{Z_0})$, $G(t)=\int_{\{\psi<-t\}}|F|^2e^{-\varphi}c(-\psi)$ for any $t\in[0,+\infty)$ and 
$H(t)=\|F\|_{OSH_{\rho}^2(M_t,\psi_t)}^2$ for a.e. $t\in(0,+\infty)$.
\end{Theorem}

Denote that 
\begin{displaymath}
	\begin{split}
		\mathcal{H}^2(M_t,\rho):=\Bigg\{\tilde f:\tilde{f}\in
H^0(\{\psi<-t\},\mathcal{O}(K_M)),\,\int_{\{\psi<-t\}}|\tilde f|^2e^{-\varphi}c(-\psi)<+\infty\Bigg\},
	\end{split}
\end{displaymath}
for $t\ge0$.

Assume that $H(t)$ is Lebesgue measurable on $(0,+\infty)$, then it is clear that $\int_0^{+\infty}H(t)dt\le G(0)$. Denote that $\hat H(t):=\int_t^{+\infty}H(s)ds$. We give a concavity property for $\hat H(t)$ in the following theorem.

\begin{Theorem}
	\label{thm3}Assume that $OSH_{\rho}^2(M_t,\psi_t)\subset \mathcal{H}^2(M_t,\rho)$ for a.e.  $t_0\in(0,+\infty)$. Then $\hat H(h^{-1}(r))$ is  concave with respect to $r\in[0,\int_0^{+\infty}c(t)e^{-t}dt]$, where   $h(t):=\int_t^{+\infty}c(s)e^{-s}ds$ for $t\ge0$. 
\end{Theorem}

\subsection{Applications: planar regions}\label{sec:appli}

In this section, we consider the case that $M=D$ is a planar region bounded by finite analytic Jordan curves.

Let $z_0\in D$, and $G_{D}(\cdot,z_0)$ be the Green function on $D$. Denote the set of all critical values of $G_{D}(\cdot,z_0)$ by $N\subset(-\infty,0)$. $N\subset\subset(-\infty,0)$ is a discrete set (see Lemma \ref{l:discrete}).
Denote that 
$$D_t:=\{z\in D:2G_D(z,z_0)<-t\}$$ for any $t\ge0$. For any $t\in[0,+\infty)\backslash -2N$, $D_t$ is a planar region bounded by finite analytic Jordan curves, thus the conjugate Hardy $H^2$ space $H^{2}_{(c)}(D_t)$ is well defined. 

We recall the conjugate Hardy $H^2$ kernel (see \cite{saitoh}) 
\begin{equation}
	\label{eq:0715a}\hat K_{D_t}(z_0):=\frac{1}{\inf\left\{\|f\|^2_{H^2_{(c)}(D_t)}:f\in H^2_{(c)}(D_t)\,\&\,f(z_0)=1\right\}},\end{equation}
where $t\in[0,+\infty)\backslash -2N$, $\|f\|^2_{H^2_{(c)}(D_t)}:=\frac{1}{2\pi}\int_{\partial D_t}|f(z)|^2\left(\frac{\partial G_D(z,z_0)}{\partial v_z}\right)^{-1}|dz|$  and $\partial/\partial v_z$ denotes the derivative along the outer normal unit vector $v_z$.
Following from Proposition \ref{l:OSH=H}, for $t\in -2N$, we define that 
$$H^2_{(c)}(D_t):=\left\{f\in\mathcal{O}(D_t):\|f\|^2_{H^2_{(c)}(D_t)}:=\liminf_{B\rightarrow0+0}\frac{\int_{\{-t-B<2G_D(\cdot,z_0)<-t\}}|f|^2d\lambda_D}{B\pi}<+\infty\right\},$$
and $\hat K_{D_t}(z_0)$ is defined as equality \eqref{eq:0715a}, where $d\lambda_D$ is the Lebesgue measure on $D$.

Denote that $B_{D_t}(\cdot,\cdot)$ is the Bergman kernel on $D_t$. Then we have 
$$\frac{2}{B_{D_t}(z_0,z_0)}=\inf\left\{\int_{D_t}|\tilde f|^2:\tilde f\in H^0(D_t,\mathcal{O}(K_D))\,\&\,\tilde f(z_0)=dz\right\}.$$
Note that $\frac{B_{D_t}(\cdot,z_0)}{B_{D_t}(z_0,z_0)}$ is a holomorphic function on $D_t$ satisfying 
$$\int_{D_t}\left|\frac{B_{D_t}(\cdot,z_0)}{B_{D_t}(z_0,z_0)}\right|^2=\inf\left\{\int_{D_t}|\tilde f|^2d\lambda_D:\tilde f\in \mathcal{O}(D_t)\,\&\,\tilde f(z_0)=1\right\},$$
then Theorem \ref{thm1} implies the following result.

\begin{Corollary}
	\label{c:1}
	The following two statements hold:
	
	$(1)$
	$\frac{\|B_{D_t}(\cdot,z_0)\|_{H^2_{(c)}(D_t)}^2}{|B_{D_t}(z_0,z_0)|^2}e^t$ is an increasing function on $[0,+\infty)$;
	
	$(2)$ equality
	$$\pi\|B_{D_t}(\cdot,z_0)\|_{H^2_{(c)}(D_t)}^2=\frac{d}{dt}B_{D_t}(z_0,z_0)$$
	holds a.e. on $(0,+\infty)$.
\end{Corollary}

Inequality \eqref{eq:0718a} implies that 
$\frac{\pi}{\hat K_{D_t}(z_0)}\le-\frac{d}{dt}\frac{1}{B_{D_t}(z_0,z_0)}.$
 The following Corollary gives a monotonicity of $\hat K_{D_t}(z_0)e^{-t}$ and a characterization for $\frac{\pi}{\hat K_{D_t}(z_0)}=-\frac{d}{dt}\frac{1}{B_{D_t}(z_0,z_0)}$ a.e. on $(0,+\infty)$.

\begin{Corollary}
	\label{c:2}$\hat K_{D_t}(z_0)e^{-t}$ is a decreasing function on $[0,+\infty)$. Moreover, the following three statements are equivalent:
	
	$(1)$ $\frac{\pi}{\hat K_{D_t}(z_0)}=-\frac{d}{dt}\frac{1}{B_{D_t}(z_0,z_0)}$ a.e. on $(0,+\infty)$;
	
	$(2)$ $\hat K_{D_t}(z_0)e^{-t}$ is a constant function on $[0,+\infty)$;
	
	$(3)$ $D$ is simply connected. 
\end{Corollary}

The following remark shows that Corollary \ref{c:2} can deduce the solution of Conjecture \ref{conj1} (the Saitoh's conjecture for conjugate Hardy $H^2$ kernels).

\begin{Remark}
	As $\hat K_{D_t}(z_0)e^{-t}$ is a decreasing function on $[0,+\infty)$ and inequality $\frac{\pi}{\hat K_{D_t}(z_0)}\le-\frac{d}{dt}\frac{1}{B_{D_t}(z_0,z_0)}$ holds for any $t\ge0$,  we have
	\begin{equation}
		\nonumber
		\begin{split}
			\frac{\pi}{\hat K_{D}(z_0)}=&\int_{0}^{{+\infty}}\frac{\pi e^{-t}}{\hat K_{D}(z_0)}dt
			\le\int_{0}^{{+\infty}}\frac{\pi}{\hat K_{D_t}(z_0)}dt\\
			\le& \int_0^{+\infty}-\frac{d}{dt}\frac{1}{B_{D_t}(z_0,z_0)}dt\le \frac{1}{B_{D}(z_0,z_0)},
		\end{split}
	\end{equation}
which shows that $\hat K_{D}(z_0)\ge\pi B_{D}(z_0,z_0)$. If the equality $\hat K_{D}(z_0)=\pi B_{D}(z_0,z_0)$ holds, then  $\hat K_{D_t}(z_0)e^{-t}$ is a constant function on $[0,+\infty)$.
Thus, using the characterization in Corollary \ref{c:2}, Conjecture \ref{conj1} has been proved. 
\end{Remark}

\section{Preparations}

In this section, we do some preparations.

\subsection{Minimal $L^2$ integrals and Ohsawa-Saitoh-Hardy space}
In this section, we recall and give some results on minimal $L^2$ integrals $G(t)$ and Ohsawa-Saitoh-Hardy space $OSH_{\rho}^2(M,\psi)$.

We firstly introduce a property of coherent analytic sheaves and a convergence property of holomorphic $(n,0)$ form.
\begin{Lemma}[see \cite{G-R}]
\label{closedness}
Let $N$ be a submodule of $\mathcal O_{\mathbb C^n,o}^q$, $1\leq q\leq \infty$, let $f_j\in\mathcal O_{\mathbb C^n}(U)^q$ be a sequence of $q-$tuples holomorphic in an open neighborhood $U$ of the origin $o$. Assume that the $f_j$ converge uniformly in $U$ towards  a $q-$tuples $f\in\mathcal O_{\mathbb C^n}(U)^q$, assume furthermore that all germs $(f_{j},o)$ belong to $N$. Then $(f,o)\in N$.	
\end{Lemma}

\begin{Lemma}
	[see \cite{GY-concavity}]\label{l:converge}
	Let $M$ be a complex manifold. Let $S$ be an analytic subset of $M$.  	
	Let $\{g_j\}_{j=1,2,...}$ be a sequence of nonnegative Lebesgue measurable functions on $M$, which satisfies that $g_j$ are almost everywhere convergent to $g$ on  $M$ when $j\rightarrow+\infty$,  where $g$ is a nonnegative Lebesgue measurable function on $M$. Assume that for any compact subset $K$ of $M\backslash S$, there exist $s_K\in(0,+\infty)$ and $C_K\in(0,+\infty)$ such that
	$$\int_{K}{g_j}^{-s_K}dV_M\leq C_K$$
	 for any $j$, where $dV_M$ is a continuous volume form on $M$.
	
 Let $\{F_j\}_{j=1,2,...}$ be a sequence of holomorphic $(n,0)$ form on $M$. Assume that $\liminf_{j\rightarrow+\infty}\int_{M}|F_j|^2g_j\leq C$, where $C$ is a positive constant. Then there exists a subsequence $\{F_{j_l}\}_{l=1,2,...}$, which satisfies that $\{F_{j_l}\}$ is uniformly convergent to a holomorphic $(n,0)$ form $F$ on $M$ on any compact subset of $M$ when $l\rightarrow+\infty$, such that
 $$\int_{M}|F|^2g\leq C.$$
\end{Lemma}

In the following, we follow the notations and assumptions of Theorem \ref{thm1}.

\begin{Lemma}[\cite{GMY-boundary2}]Let $B \in (0, +\infty)$ and $t_0>t_1 > 0$ be arbitrarily given. Let F be a holomorphic $(n,0)$ form on $\{\psi< -t_0\}$ such that
\begin{equation}
\int_{\{\psi<-t_0\}} {|F|}^2e^{-\varphi}c(-\psi)<+\infty,
\label{condition of lemma 2.2}
\end{equation}
Then there exists a holomorphic $(n,0)$ form $\tilde{F}$ on $\{\psi<-t_1\}$ such that
\begin{equation}
\int_{\{\psi<-t_1\}}|\tilde{F}-(1-b_{t_0,B}(\psi))F|^2e^{-\varphi-\psi+v_{t_0,B}(\psi)}c(-v_{t_0,B}(\psi))
\le C\int_{t_1}^{t_0+B}c(t)e^{-t}dt,
\end{equation}
where $C=\int_M \frac{1}{B} \mathbb{I}_{\{-t_0-B< \psi < -t_0\}}  {|F|}^2
e^{{-}\varphi-\psi}$, $b_{t_0,B}(t)=\int^{t}_{-\infty}\frac{1}{B} \mathbb{I}_{\{-t_0-B< s < -t_0\}}ds$ and
$v_{t_0,B}(t)=\int^{t}_{-t_0}b_{t_0,B}(s)ds-t_0$.
\label{lem:GZ_sharp2}
\end{Lemma}

The following lemma gives a characterization for the minimal $L^2$ integral $G(t)$ being equal to $0$.

\begin{Lemma}[see \cite{GMY-boundary2}]\label{lem:0}
$f\in H^0(Z_0,(\mathcal{O}(K_{M})\otimes\mathcal{F})|_{Z_0})\Leftrightarrow G(t)=0$.
\end{Lemma}

We recall a concavity property for the minimal $L^2$ integral $G(t)$.
\begin{Theorem}[\cite{GMY-boundary2}]$G(h^{-1}(r))$ is concave with respect to  $r\in (0,\int_{0}^{+\infty}c(t)e^{-t}dt)$, $\lim\limits_{t\to 0+0}G(t)=G(0)$ and $\lim\limits_{t \to +\infty}G(t)=0$.
\label{thm:concave}
\end{Theorem}

The following lemma shows the existence and uniqueness of the minimal holomorphic form $F_t$.
\begin{Lemma}[see \cite{GMY-boundary2}]\label{lem:A}
Assume that $G(t)<+\infty$ for some $t\in[T,+\infty)$.
Then there exists a unique holomorphic $(n,0)$ form $F_{t}$ on
$\{\psi<-t\}$ satisfying $(F_{t}-f)\in H^{0}(Z_0,(\mathcal{O}(K_{M})\otimes\mathcal{F})|_{Z_0})$ and $\int_{\{\psi<-t\}}|F_{t}|^{2}e^{-\varphi}c(-\psi)=G(t)$.
Furthermore,
for any holomorphic $(n,0)$ form $\hat{F}$ on $\{\psi<-t\}$ satisfying $(\hat{F}-f)\in H^{0}(Z_0,(\mathcal{O}(K_{M})\otimes\mathcal{F})|_{Z_0})$ and
$\int_{\{\psi<-t\}}|\hat{F}|^{2}e^{-\varphi}c(-\psi)<+\infty$,
we have the following equality
\begin{equation}
\label{equ:20170913e}
\begin{split}
&\int_{\{\psi<-t\}}|F_{t}|^{2}e^{-\varphi}c(-\psi)+\int_{\{\psi<-t\}}|\hat{F}-F_{t}|^{2}e^{-\varphi}c(-\psi)
\\=&
\int_{\{\psi<-t\}}|\hat{F}|^{2}e^{-\varphi}c(-\psi).
\end{split}
\end{equation}
\end{Lemma}

Let $F_t$ be the unique holomorphic $(n,0)$ form on $M_t$ satisfying that $(F_t-f)\in H^0(Z_0 ,(\mathcal{O} (K_M) \otimes \mathcal{F})|_{Z_0})$ and $G(t)=\int_{ \{ \psi<-t\}}|F_t|^2e^{-\varphi}c(-\psi)$. By definitions, we have $F_t\in OSH_{\rho}^2(M_t,\psi_t)$ and 
\begin{equation}
	\label{eq:0522a2}\|F_t\|_{OSH_{\rho}^2(M_t,\psi_t)}^2=\liminf_{B\rightarrow0+0}\frac{\int_{\{-t-B<\psi<-t\}}|F_t|^2e^{-\varphi}c(-\psi)}{B}\leq-G'_{+}(t).
\end{equation}
Following the concavity of $G(h^{-1}(r))$, we have 
\begin{equation}
	\label{eq:0716a}G(t_1)\le G(t_0)+\frac{\int_{t_1}^{t_0}c(s)e^{-s}ds}{c_+(t_0)e^{-t_0}}(-G'_{+}(t_0))
\end{equation}
for any $0\le t_1<t_0<+\infty$. The following proposition shows that inequality \eqref{eq:0716a} also holds when replacing $-G'_{+}(t)$ by $\|F_t\|_{OSH_{\rho}^2(M_t,\psi_t)}^2$. 

\begin{Proposition}
\label{p:1} For any $t_0>t_1\ge0$, inequality
	\begin{equation}
		\label{eq:0522b}
		G(t_1)\le G(t_0)+\frac{\int_{t_1}^{t_0}c(s)e^{-s}ds}{c_+(t_0)e^{-t_0}}\|F_{t_0}\|_{OSH_{\rho}^2(M_{t_0},\psi_t)}^2
	\end{equation}
	holds,
	where $c_+(t)=\lim_{s\rightarrow t+0}c(s)$.	
\end{Proposition}

\begin{proof}

Lemma \ref{lem:GZ_sharp2}
shows that for any $B>0$,
there exists a
holomorphic $(n,0)$ form $\tilde{F}_{B}$ on $\{\psi<-t_1\}$, such that
$(\tilde{F}_{B}-F_{t_{0}})\in H^{0}(Z_0,(\mathcal{O}(K_M)\otimes\mathcal{I}(\varphi+\psi))|_{Z_0})
\subseteq H^{0}(Z_0,(\mathcal{O}(K_{M})\otimes\mathcal{F})|_{Z_0})$
and
\begin{equation}
\label{equ:GZc}
\begin{split}
&\int_{\{\psi<-t_1\}}|\tilde{F}_{B}-(1-b_{t_{0},B}(\psi))F_{t_{0}}|^{2}e^{-\varphi-\psi+v_{t_0,B}(\psi)}c(-v_{t_0,B}(\psi))
\\\leq&
\int_{t_1}^{t_{0}+B}c(t)e^{-t}dt
\int_{\{\psi<-t_1\}}\frac{1}{B}(\mathbb{I}_{\{-t_{0}-B<\psi<-t_{0}\}})|F_{t_{0}}|^{2}e^{-\varphi-\psi}
\\\leq&
\frac{e^{t_{0}+B}\int_{t_1}^{t_{0}+B}c(t)e^{-t}dt}{\inf_{t\in(t_{0},t_{0}+B)}c(t)}
\int_{\{\psi<-t_1\}}\frac{1}{B}(\mathbb{I}_{\{-t_{0}-B<\psi<-t_{0}\}})|F_{t_{0}}|^{2}e^{-\varphi}c(-\psi).
\end{split}
\end{equation}
As $t\leq v_{t_0,B}(t)$, the decreasing property of $c(t)e^{-t}$ shows that
$$e^{-\psi+v_{t_0,B}(\psi)}c(-v_{t_0,B}(\psi))\geq c(-\psi).$$
Inequality \eqref{equ:GZc} and \eqref{eq:0522a2} imply that 
\begin{equation}
	\label{eq:0610a}
	\begin{split}
		&\liminf_{B\rightarrow0+0}\int_{\{\psi<-t_1\}}|\tilde{F}_{B}-(1-b_{t_{0},B}(\psi))F_{t_{0}}|^{2}e^{-\varphi}c(-\psi)\\
		\le &\liminf_{B\rightarrow0+0}\int_{\{\psi<-t_1\}}|\tilde{F}_{B}-(1-b_{t_{0},B}(\psi))F_{t_{0}}|^{2}e^{-\varphi-\psi+v_{t_0,B}(\psi)}c(-v_{t_0,B}(\psi))\\
		\le& \frac{\int_{t_1}^{t_0}c(s)e^{-s}ds}{c_+(t_0)e^{-t_0}}\|F_{t_0}\|_{OSH_{\rho}^2(M_{t_0},\psi_t)}^2\\
		<&+\infty.
	\end{split}
\end{equation}

Note that
\begin{equation}
\label{equ:GZd}
\begin{split}
&\left(\int_{\{\psi<-t_1\}}|\tilde{F}_{B}-(1-b_{t_{0},B}(\psi))F_{t_{0}}|^{2}e^{-\varphi}c(-\psi)\right)^{\frac{1}{2}}
\\\geq&\left(\int_{\{\psi<-t_1\}}|\tilde{F}_{B}|^{2}e^{-\varphi}c(-\psi)\right)^{\frac{1}{2}}-\left(\int_{\{\psi<-t_1\}}|(1-b_{t_{0},B}(\psi))F_{t_{0}}|^{2}e^{-\varphi}c(-\psi)\right)^{\frac{1}{2}}\\
\ge&\left(\int_{\{\psi<-t_1\}}|\tilde{F}_{B}|^{2}e^{-\varphi}c(-\psi)\right)^{\frac{1}{2}}-\left(\int_{\{\psi<-t_0\}}|F_{t_{0}}|^{2}e^{-\varphi}c(-\psi)\right)^{\frac{1}{2}},
\end{split}
\end{equation}
then it follows from inequality \eqref{eq:0610a} that
\begin{equation}
\nonumber
\begin{split}
&\liminf_{B\rightarrow0+0}\left(\int_{\{\psi<-t_1\}}|\tilde{F}_{B}|^{2}e^{-\varphi}c(-\psi)\right)^{\frac{1}{2}}
\\\leq&\frac{\int_{t_1}^{t_0}c(s)e^{-s}ds}{c_+(t_0)e^{-t_0}}\|F_{t_0}\|_{OSH_{\rho}^2(M_{t_0},\psi_t)}^2+\left(\int_{\{\psi<-t_0\}}|F_{t_{0}}|^{2}e^{-\varphi}c(-\psi)\right)^{\frac{1}{2}}.
\end{split}
\end{equation}
As $c\in P_{0,M}$,
it follows from  Lemma \ref{l:converge} that there exists a subsequence of $\{\tilde{F}_B\}$, denoted by $\{\tilde{F}_{B_k}\}_{k\in\mathbb{N}^+}$ ($B_k\rightarrow0$ when $k\rightarrow+\infty$), which is uniformly convergent to a holomorphic $(n,0)$ form $F_1$ on $\{\psi<-t_1\}$ on any compact subset of $\{\psi<-t_1\}$ when $k\rightarrow+\infty$,  such that
$$\int_{\{\psi<-t_1\}}|F_1|^2e^{-\varphi}c(-\psi)\le\liminf_{B\rightarrow0+0}\int_{\{\psi<-t_1\}}|\tilde{F}_{B}|^{2}e^{-\varphi}c(-\psi)<+\infty.$$  As $(\tilde{F}_B-F_{t_{0}})\in H^{0}(Z_0,(\mathcal{O}(K_{M})\otimes\mathcal{F})|_{Z_0})$ for any $B>0$, we have $(F_1-F_{t_{0}})\in H^{0}(Z_0,(\mathcal{O}(K_{M})\otimes\mathcal{F})|_{Z_0})$ by Lemma \ref{closedness}.
Note that 
\begin{equation*}
	\lim_{k\rightarrow+\infty}b_{t_0,B_k}(t)=\lim_{k\rightarrow+\infty}\int_{-\infty}^{t}\frac{1}{B_k}\mathbb{I}_{\{-t_0-B_k<s<-t_0\}}ds=\left\{ \begin{array}{lcl}
	0 & \mbox{if}& x\in(-\infty,-t_0)\\
	1 & \mbox{if}& x\in[-t_0,+\infty)
    \end{array} \right.
\end{equation*}
and
\begin{equation*}
	\lim_{j\rightarrow+\infty}v_{t_0,B_k}(t)=\lim_{j\rightarrow+\infty}\int_{-t_0}^{t}b_{t_0,B_k}ds-t_0=\left\{ \begin{array}{lcl}
	-t_0 & \mbox{if}& x\in(-\infty,-t_0)\\
	t & \mbox{if}& x\in[-t_0,+\infty)
    \end{array}. \right.
\end{equation*}
Following from  inequality \eqref{eq:0610a} and the Fatou's Lemma, we have 
\begin{equation}
	\nonumber
	\begin{split}
		&\int_{\{\psi<-t_0\}}|F_1-F_{t_0}|^2e^{-\varphi}c(-\psi)+\int_{\{-t_0\le\psi<-t_1\}}|F_1|^2e^{-\varphi}c(-\psi)\\
		=&\int_{\{\psi<-t_1\}}\lim_{k\rightarrow+\infty}|\tilde{F}_{B}-(1-b_{t_{0},B}(\psi))F_{t_{0}}|^{2}e^{-\varphi}c(-\psi)
\\\le&\liminf_{B\rightarrow0+0}\int_{\{\psi<-t_1\}}|\tilde{F}_{B}-(1-b_{t_{0},B}(\psi))F_{t_{0}}|^{2}e^{-\varphi}c(-\psi)\\
		\le &\frac{\int_{t_1}^{t_0}c(s)e^{-s}ds}{c_+(t_0)e^{-t_0}}\|F_{t_0}\|_{OSH_{\rho}^2(M_{t_0},\psi_t)}^2.
	\end{split}
\end{equation}
Using Lemma \ref{lem:A}, we get that that 
\begin{equation}
	\nonumber
	\begin{split}
	G(t_1)&\leq\int_{\{\psi<-t_1\}}|F_1|^2e^{-\varphi}c(-\psi)\\
	&=\int_{\{-t_0\le\psi<-t_1\}}|F_1|^2e^{-\varphi}c(-\psi)+\int_{\{\psi<-t_0\}}|F_1-F_{t_0}|^2e^{-\varphi}c(-\psi)+G(t_0)\\
	&\le \frac{\int_{t_1}^{t_0}c(s)e^{-s}ds}{c_+(t_0)e^{-t_0}}\|F_{t_0}\|_{OSH_{\rho}^2(M_{t_0},\psi_t)}^2+G(t_0).
	\end{split}
\end{equation}

Thus, Proposition \ref{p:1} holds.
\end{proof}

Recall that 
	\begin{equation}\nonumber
\begin{split}
H(t):=\inf\Bigg\{\|\tilde f\|_{OSH_{\rho}^2(M_t,\psi_t)}^2: \tilde{f}\in
H^0(\{\psi<-t\},\mathcal{O}(K_M)) \\
\&\, (\tilde{f}-f)\in
H^0(Z_0 ,(\mathcal{O} (K_M) \otimes \mathcal{F})|_{Z_0})\Bigg\}
\end{split}
\end{equation}	
for any $t\ge0$, then we have $H(t)\le -G_+'(t)$. Following the concavity of $G(h^{-1}(r))$, we have 
\begin{equation}
	\label{eq:0716b}\int_{t_1}^{t_0}-G'_+(t)dt=G(t_1)-G(t_0)\le\frac{\int_{t_1}^{t_0}c(s)e^{-s}ds}{c_+(t_0)e^{-t_0}}(-G'_{+}(t_0))
\end{equation}
 for any $0\le t_1<t_0<+\infty$. The following proposition shows that when replacing $-G_+'(t)$ by $H(t)$, inequality \eqref{eq:0716b} also holds under some assumptions.

\begin{Proposition}
	\label{p:2}
	Let $t_0>0$. Assume that  $OSH_{\rho}^2(M_{t_0},\psi_{t_0})\subset \mathcal{H}^2(M_{t_0},\rho)$ and $H(t)$ is Lebesgue measurable on $(0,+\infty)$. Then we have
	\begin{equation}
		\label{eq:0613a}
		\frac{\int_{t_1}^{t_0}H(t)dt}{\int_{t_1}^{t_0}c(t)e^{-t}dt}\le \frac{e^{t_0}}{c_+(t_0)}H(t_0),
	\end{equation}
	where $t_1\in[0,t_0)$.  
\end{Proposition}
\begin{proof}Without loss of generality, we assume $H(t_0)<+\infty$.
	By definition of $H(t_0)$, there exists $\{f_j\}_{j\in\mathbb{Z}_{>0}}\subset OSH_{\rho}^2(M_{t_0},\psi_{t_0})\subset \mathcal{H}^2(M_{t_0},\rho)$ such that 
	$(f_j-f)\in H^0(Z_0,(\mathcal{O}(K_M)\otimes\mathcal{F})|_{Z_0})$ for any $j$ and 
	\begin{equation}
		\label{eq:0615a}
		H(t_0)=\lim_{j\rightarrow+\infty}\|f_j\|^2_{OSH_{\rho}^2(M_{t_0},\psi_{t_0})}.
	\end{equation}
	
Lemma \ref{lem:GZ_sharp2}
shows that for any $B>0$ and $j>0$,
there exists a 
holomorphic $(n,0)$ form $\tilde{F}_{j,B}$ on $\{\psi<-t_1\}$, such that
$(\tilde{F}_{j,B}-f_j)\in H^{0}(Z_0,(\mathcal{O}(K_M)\otimes\mathcal{I}(\varphi+\psi))|_{Z_0})
\subseteq H^{0}(Z_0,(\mathcal{O}(K_{M})\otimes\mathcal{F})|_{Z_0})$
and
\begin{equation}
\label{eq:0615b}
\begin{split}
&\int_{\{\psi<-t_1\}}|\tilde{F}_{j,B}-(1-b_{t_{0},B}(\psi))f_j|^{2}e^{-\varphi-\psi+v_{t_0,B}(\psi)}c(-v_{t_0,B}(\psi))
\\\leq&
\int_{t_1}^{t_{0}+B}c(t)e^{-t}dt
\int_{\{\psi<-t_1\}}\frac{1}{B}(\mathbb{I}_{\{-t_{0}-B<\psi<-t_{0}\}})|f_j|^{2}e^{-\varphi-\psi}
\\\leq&
\frac{e^{t_{0}+B}\int_{t_1}^{t_{0}+B}c(t)e^{-t}dt}{\inf_{t\in(t_{0},t_{0}+B)}c(t)}
\int_{\{\psi<-t_1\}}\frac{1}{B}(\mathbb{I}_{\{-t_{0}-B<\psi<-t_{0}\}})|f_j|^{2}e^{-\varphi}c(-\psi).
\end{split}
\end{equation}
As $H(t)\le \liminf_{\epsilon\rightarrow0+0}\frac{1}{\epsilon}\int_{\{-t-\epsilon<\psi<-t\}}|\tilde F_{j,B}|^2e^{-\varphi}c(-\psi)$ and $v_{t_0,B}(t)=t$ for $t\in[-t_0,-t_1)$, inequality \eqref{eq:0615b} implies that
\begin{equation}
	\label{eq:0615c}
	\begin{split}
			\int_{t_1}^{t_0}H(t)dt\le &\int_{\{-t_0\le\psi<-t_1\}}|\tilde F_{j,B}|^2e^{-\varphi}c(-\psi)\\
			\le&\int_{\{\psi<-t_1\}}|\tilde{F}_{j,B}-(1-b_{t_{0},B}(\psi))f_j|^{2}e^{-\varphi-\psi+v_{t_0,B}(\psi)}c(-v_{t_0,B}(\psi))\\
			\le &\frac{e^{t_{0}+B}\int_{t_1}^{t_{0}+B}c(t)e^{-t}dt}{\inf_{t\in(t_{0},t_{0}+B)}c(t)}
\int_{\{\psi<-t_1\}}\frac{1}{B}(\mathbb{I}_{\{-t_{0}-B<\psi<-t_{0}\}})|f_j|^{2}e^{-\varphi}c(-\psi).
	\end{split}
\end{equation}
Letting $B\rightarrow0+0$ and $j\rightarrow+\infty$, equality \eqref{eq:0615a} and inequality \eqref{eq:0615c} show
$$\frac{\int_{t_1}^{t_0}H(t)dt}{\int_{t_1}^{t_0}c(t)e^{-t}dt}\le \frac{e^{t_0}}{c_+(t_0)}H(t_0).$$

Thus, Proposition \ref{p:2} holds.	
\end{proof}

\subsection{Conjugate Hardy $H^2$ spaces on planar regions}In this section, we recall and give some results on the conjugate Hardy $H^2$ spaces on planar regulars.

Let $D$ be a planar regular region bounded by finite analytic Jordan curves (see \cite{saitoh,yamada}).  Let $z_0\in D$, and $G_{D}(\cdot,z_0)$ be the Green function on $D$.
The definition of $H^{2}_{(c)}(D)$ (see \cite{saitoh}) can be referred to Section \ref{sec:appli}.
The following lemma gives some properties for $H^2_{(c)}(D)$.
\begin{Lemma}[\cite{rudin55}]
\label{l:0-1}
	$(a)$ If $f\in H^2_{(c)}(D)$, there is a function $f_*$ on $\partial D$ such that $f$  has nontangential boundary value $f_*$ almost everywhere on $\partial D$. And the map $\gamma:f\longmapsto f_*\in L^2(\partial D)$ satisfies that
	$$u_f(z_0)=\frac{1}{2\pi}\int_{D}|\gamma(f)|^2\frac{\partial G_{D}(z,z_0)}{\partial v_z}|dz|$$
	holds for any $f\in H^2_{(c)}(D)$, where $u_f$ is the least harmonic majorant of $|f|^2$. 
	
	$(b)$ for any $g\in L^2(\partial D)$, $g\in \gamma(H_{(c)}^2(D))$ if and only if 
	$$\int_{\partial D}g(z)\phi(z)dz=0$$
	holds for any holomorphic function $\phi$ on a neighborhood of  $\overline D$.
	 
	 $(c)$ The inverse of $\gamma $ is given by 
	\begin{equation}
		\label{eq:0728a}f(w)=\frac{1}{2\pi\sqrt{-1}}\int_{\partial D}\frac{f_*(z)}{z-w}dz
	\end{equation}
	for any $z\in D$.
\end{Lemma}

Each function $f(z)\in H_{(c)}^2(D)$ has Fatou's nontangential boundary value a.e. on $\partial D$ belonging to $L^2(\partial D)$, then $f\in  H_{(c)}^2(D)$ can be seen as a function on $\overline D$ without misunderstanding.

 We recall a basic formula, which will be used in the subsequent discussion.
\begin{Lemma}[see \cite{GY-saitohprodct}]
	\label{l:partial}
	Let $\psi\in C^1(U)$ satisfy $\psi|_{\partial D}=0$ and $|d\psi|\not=0$ on $\partial D$, where $U$ is a neighborhood of $\partial D$. Then 
	$\frac{\partial \psi}{\partial v_z}=\left(\left(\frac{\partial \psi}{\partial x}\right)^2+\left(\frac{\partial \psi}{\partial y}\right)^2\right)^{\frac{1}{2}}$ on $\partial D$, where $\partial/\partial v_z$ denotes the derivative along the outer normal unit vector $v_z$ and $z=(x,y)$. 
\end{Lemma}

 We recall the following coarea formula.
\begin{Lemma}[see \cite{federer}]
	\label{l:coarea}Suppose that $\Omega$ is an open set in $\mathbb{R}^n$ and $u\in C^1(\Omega)$. Then for any $g\in L^1(\Omega)$, 
	$$\int_{\Omega}g(x)|\nabla u(x)|dx=\int_{\mathbb{R}}\left(\int_{u^{-1}(t)}g(x)dH_{n-1}(x)\right)dt,$$
	where $H_{n-1}$ is the $(n-1)$-dimensional Hausdorff measure.
\end{Lemma}

Denote that 
$$D_t:=\{z\in D:2G_D(z,z_0)<-t\}$$ for any $t\ge0$. The following two lemmas will be used in the proof of Proposition \ref{l:OSH=H}. 
\begin{Lemma} [see \cite{GY-saitohfinite}] \label{l:0-4v2}Let $\varphi$ be a positive Lebesgue measurable function on $U\cap \overline D$ satisfying that  $\lim_{z\rightarrow\tilde z}\varphi(z)=\varphi(\tilde z)$ for any $\tilde z\in\partial D$. Then 
\begin{equation}
	\nonumber
	\int_{\partial D}|f|^2\varphi|dz|=\lim_{t\rightarrow0+0}\int_{\partial D_t}|f|^2\varphi|dz|
\end{equation}
	holds for any $f\in H_{(c)}^2(D)$.
\end{Lemma}

\begin{Lemma}[see \cite{GY-weightedsaitoh}]
	\label{l:01a}
	Let $f$ be a holomorphic function on $D$.   Assume that 
	\begin{equation}
		\nonumber \liminf_{r\rightarrow1-0}\frac{\int_{\{z\in D:2G_{D}(z,z_0)\ge\log r\}}|f(z)|^2d\lambda_D}{1-r}<+\infty,
	\end{equation}
	then we have $f\in H_{(c)}^{2}(D)$, where $d\lambda_D$ is the Lebesgue measure on $D$.
\end{Lemma}

   Denote the set of all critical values of $G_{D}(\cdot,z_0)$ by $N\subset(-\infty,0)$. 

\begin{Lemma}\label{l:discrete}
	$N\subset\subset(-\infty,0)$ is a discrete set.
\end{Lemma}
\begin{proof}
	As $D$ is a planar region bounded by finite analytic Jordan curves, $G_{D}(\cdot,z_0)$ can be extended to a harmonic function on $U\cup D\backslash\{z_0\}$ and $|\nabla G_{D}(\cdot,z_0)|\not=0$ on $\partial D$, where $U$ is a neighborhood of $\partial D$. Note that $G_{D}(\cdot,z_0)-\log|z-z_0|$ is harmonic near $z_0$. Hence we have $N\subset\subset(-\infty,0)$. Note that $G_{D}(\cdot,z_0)$ is harmonic on $D\backslash\{z_0\}$, then the set of all critical points $G_{D}(\cdot,z_0)$ is a discrete subset of $D\backslash\{z_0\}$. Thus, we have $N\subset\subset(-\infty,0)$ is a discrete set.
\end{proof}
 
  For any $t\in[0,+\infty)\backslash -2N$, $D_t$ is a planar region bounded by finite analytic Jordan curves. The following proposition shows that for planar region $D_t$ and trivial weight $\rho\equiv1$, the Ohsawa-Saitoh-Hardy space is just the conjugate Hardy $H^2$ space.

\begin{Proposition}
	\label{l:OSH=H}
	Let $t\in[0,+\infty)\backslash-2N$. Then 
	$$OSH^2(D_t,2G_{D}(\cdot,z_0)+t)=\{fdz:f\in H^2_{(c)}(D_t)\},$$ and 
	$$\|fdz\|^2_{OSH^2(D_t,2G_{D}(\cdot,z_0)+t)}=2\pi\|f\|_{H^2_{(c)}(D_t)}^2$$
	for any $f\in H^2_{(c)}(D_t),$ where $\|f\|_{H^2_{(c)}(D)}:=\left(\frac{1}{2\pi}\int_{\partial D}|f(z)|^2\left(\frac{\partial G_{D}(z,z_0)}{\partial v_z} \right)^{-1}|dz|\right)^{\frac{1}{2}}$.
\end{Proposition}
\begin{proof}
	As $D$ is a planar regular region bounded by finite analytic Jordan curves and $G_D(\cdot,z_0)$ is harmonic on $D\backslash\{z_0\}$, it suffices to prove the case $t=0$.
	Let $f$ be any holomorphic function on $D$. 
	
	If $fdz\in OSH^2(D,2G_{D}(\cdot,z_0))$, then we have 
	\begin{equation}
		\nonumber
		\begin{split}
		\liminf_{B\rightarrow 0+0}\frac{1}{B}\int_{\{-B\le\psi<0\}}|f|^2d\lambda_D 	&=\liminf_{B\rightarrow 0+0}\frac{1}{2B}\int_{\{-B<\psi<0\}}|fdz|^2 \\
		&=\frac{1}{2}\|f\|^2_{OSH^2(D,\psi)}\\
		&<+\infty,
		\end{split}
	\end{equation}
	where $d\lambda_D$ is the Lebesgue measure on $D$ and $\psi=2G_{D}(z,z_0)$, which implies that $f\in H^2_{(c)}(D)$ by Lemma \ref{l:01a}.
	
	In the following, assume that $f\in H^2_{(c)}(D)$.
Using Lemma \ref{l:coarea} and Lemma \ref{l:0-4v2}, we have 
\begin{equation}
	\nonumber
	\begin{split}
		\lim_{B\rightarrow 0+0}\frac{1}{B}\int_{\{-B<\psi<0\}}|f|^2d\lambda_D&=\lim_{B\rightarrow 0+0}\frac{\int_{0}^{B}\left(\int_{\{\psi=-t\}}|f|^2|\nabla \psi|^{-1}|dz|\right)dt}{B}\\
		&=\frac{1}{2}\int_{\partial D}|f(z)|^2\left(\frac{\partial G_D(z,z_0)}{\partial v_z}\right)^{-1}|dz|,
	\end{split}
\end{equation}
which implies that $f\in OSH^2(D,2G_{D}(z,z_0))$ and 
$$2\pi\|f\|_{H^2_{(c)}(D)}^2=\int_{\partial D}|f(z)|^2\left(\frac{\partial G_D(z,z_0)}{\partial v_z}\right)^{-1}|dz|=\|fdz\|^2_{OSH^2(D,2G_{D}(z,z_0))}.$$

Thus, Proposition \ref{l:OSH=H} holds.
\end{proof}

According to the above proposition, for $t\in-2N$, we define that $$H^2_{(c)}(D_t):=\{f\in\mathcal{O}(D_t):fdz\in OSH^2(D_t,2G_{D}(\cdot,z_0)+t)\}$$
and 
$$\|f\|^2_{H^2_{(c)}(D_t)}:=\frac{1}{2\pi}\|fdz\|^2_{OSH^2(D_t,2G_{D}(z,z_0)+t)}$$
for any $f\in H^2_{(c)}(D_t)$.
For any $t\ge0$, The conjugate Hardy $H^2$ kernel is defined by
$$\hat K_{D_t}(z_0):=\frac{1}{\inf\left\{\|f\|^2_{H^2_{(c)}(D_t)}:f\in H^2_{(c)}(D_t)\,\&\,f(z_0)=1\right\}}.$$

We recall an existence and uniqueness property for $\hat K_{D_t}(z_0)$.
\begin{Lemma}
	[see \cite{GY-saitohfinite}]
	\label{l:uni}
	For any $t\in [0,+\infty)\backslash-2N,$ there exists a unique holomorphic function $f_t\in H^2_{(c)}(D_t)$ such that $f_t(z_0)=1$ and $\|f_t\|^2_{H_{(c)}^2(D_t)}=\frac{1}{\hat K_{D_t}(z_0)}$.
\end{Lemma}

The following lemma will be used in the proof of Lemma \ref{l:l-continous}.

\begin{Lemma}
	\label{l:extext}Let $\rho_1\in C([-1,1]\times[0,1])$ be a positive function. Let $\{t_j\}_{j\in\mathbb{Z}_{\ge1}}\in (0,1)$ be a decreasing  sequence satisfying $\lim_{j\rightarrow+\infty}t_j=0$, and let $f_{t_j}\in L^2[0,1]$ for any $j\ge1$, which satisfies that 
	\begin{equation}
		\label{eq:0630a}
		T:=\lim_{j\rightarrow+\infty}\int_0^1|f_{t_j}(x)|^2\rho_1(t_j,x)dx<+\infty.
	\end{equation}
	Then there exists a subsequence of $\{f_{t_j}\}$ (denoted also by $\{f_{t_j}\}$), which weakly converges to $f_0\in L^2[0,1]$ and satisfies that 
	$$\int_0^1|f_0|^2\rho_1(0,x)dx\le T$$
	and 
	$$\int_0^1f_0(x)\rho_2(0,x)dx=\lim_{j\rightarrow+\infty}\int_0^1f_{t_j}(x)\rho_2(t_j,x)dx$$
	for any $\rho_2\in C([-1,1]\times[0,1]).$
\end{Lemma}
\begin{proof}
	As $\rho_1>0$ and $\rho_1\in C([-1,1]\times[0,1])$, inequality \eqref{eq:0630a} implies $$\lim_{j\rightarrow+\infty}\int_0^1|f_{t_j}(x)|^2dx<+\infty.$$
	Since $L^2[0,1]$ is a Hilbert space, there exists a subsequence of $\{f_{t_j}\}$ denoted also by $\{f_{t_j}\}$ such that $\{f_{t_j}\}$ weakly converges to $f_0\in L^2[0,1]$ and $\{\rho_1(t_j,\cdot)^{\frac{1}{2}}f_{t_j}\}$ weakly converges to $\tilde f_0\in L^2[0,1]$.
	
	For any $\rho_2\in C([-1,1]\times[0,1])$, we have 
	\begin{equation}
		\label{eq:0630b}
		\begin{split}
			&\lim_{j\rightarrow+\infty}\int_0^1f_{t_j}\rho_2(t_j,x)dx\\
			=&\lim_{j\rightarrow+\infty}\int_0^1f_{t_j}(x)\left(\rho_2(t_j,x)-\rho_2(0,x)\right)dx+\lim_{j\rightarrow+\infty}\int_0^1f_{t_j}(x)\rho_2(0,x)dx\\
			=&\int_0^1f_{0}(x)\rho_2(0,x)dx.
		\end{split}
	\end{equation}
For any $g\in C[0,1]$, equality \eqref{eq:0630b} deduces that 
\begin{equation}
	\nonumber
\lim_{j\rightarrow+\infty}\int_0^1f_{t_j}(x)\rho_1(t_j,x)^{\frac{1}{2}}g(x)dx
		=\int_0^1f_{0}(x)\rho_1(0,x)^{\frac{1}{2}}g(x)dx,
\end{equation}	
	which shows that $\tilde f_0=f_{0}(x)\rho_1(0,x)^{\frac{1}{2}}$ and 
	$$\int_0^1|f_0|^2\rho_1(0,x)dx\le\lim_{j\rightarrow+\infty}\int_0^1|f_{t_j}(x)|^2\rho_1(t_j,x)dx.$$
	
	Thus, Lemma \ref{l:extext} holds.
\end{proof}

The following lemma shows that $\hat K_{D_t}(z_0)$ is right continuous.

\begin{Lemma}
	\label{l:l-continous}
	For any $t_0\ge0$, we have
	$\lim_{t\rightarrow t_0+0}\hat K_{D_t}(z_0)=\hat K_{D_{t_0}}(z_0).$
\end{Lemma}
\begin{proof}
Firstly, we consider the case $t_0\in [0,+\infty)\backslash-2N$.

	For any $t_0\in [0,+\infty)\backslash-2N,$ let $f_{t_0}$ be the unique holomorphic function such that $f_{t_0}(z_0)=1$ and 
	\begin{equation}
		\label{eq:0629a}
		\|f_{t_0}\|^2_{H_{(c)}^2(D_t)}=\frac{1}{\hat K_{D_{t_0}}(z_0)}
	\end{equation}
	 by Lemma \ref{l:uni}. Following from  Lemma \ref{l:0-4v2}, we have 
	 \begin{equation}
	 	\label{eq:0629b}
	 	\begin{split}
	 		 	\|f_{t_0}\|^2_{H_{(c)}^2(D_{t_0})}=\frac{1}{2\pi}\lim_{t\rightarrow t_0+0}\int_{\{2G_D(\cdot,z_0)=-t\}}|f_{t_0}|^2|\nabla G_D(\cdot,z_0)|^{-1}|dz|.
	 	\end{split}
	 \end{equation} 
	 As $\frac{1}{2\pi}\int_{\{2G_D(\cdot,z_0)=-t\}}|f_{t_0}|^2|\nabla G_D(\cdot,z_0)|^{-1}|dz|\ge \frac{1}{\hat K_{D_t}(z_0)}$, it follows from equality \eqref{eq:0629a} and \eqref{eq:0629b} that 
	 \begin{equation}
	 	\label{eq:0619c}
	 	\liminf_{t\rightarrow t_0+0}\hat K_{D_t}(z_0)\ge\hat K_{D_{t_0}}(z_0).
	 \end{equation}
	 
	Then, we will prove $\limsup_{t\rightarrow t_0+0}\hat K_{D_t}(z_0)\le\hat K_{D_{t_0}}(z_0).$ It suffices to consider the case $\limsup_{t\rightarrow t_0+0}\hat K_{D_t}(z_0)>0$.  There exists a decreasing sequence $\{t_j\}_{j\in \mathbb{Z}_{\ge1}}$ such that $\lim_{j\rightarrow+\infty}t_j=t_0$ and
	 \begin{equation}
	 	\label{eq:06299b}
	 	\limsup_{t\rightarrow t_0+0}\hat K_{D_t}(z_0)=\lim_{j\rightarrow+\infty}\hat K_{D_{t_j}}(z_0).
	 \end{equation}
Without loss of generality, assume that $|\nabla G_D(\cdot,z_0)|\not=0$ on $\{-t_1<2G_D(\cdot,z_0)<-t_0\}$. For any $j\ge1$, it follows from Lemma \ref{l:0-1} that  
$$\int_{\partial D_{t_j}}\frac{f_{t_j}(z)}{z-w}dz=f_{t_j}(w)$$
for any $w\in D_{t_j}$
and
$$\int_{\partial D_{t_j}}f_{t_j}(z)\phi(z)dz=0$$
 for any holomorphic function $\phi$ on a neighborhood of $\overline D_{t_j}$.
 Note
that $$\lim_{j\rightarrow+\infty}\frac{1}{2\pi}\int_{\partial D_{t_j}}|f_{t_j}|^2\left(\frac{\partial G_D(z,z_0)}{\partial v_z} \right)^{-1}|dz|=\lim_{j\rightarrow+\infty}\frac{1}{\hat K_{D_j}(z_0)}<+\infty.$$
Dividing $\{z\in D:-t_1<2G_D(z,z_0)<-t_0\}$ into finite small domains  and using Lemma \ref{l:extext}, there exist a subsequence of $\{f_{t_j}\}$ (denoted also by $\{f_{t_j}\}$) and an element $f_0\in L^2(\partial D_{t_0})$ satisfying that 
$$\int_{\partial D_{t_0}}\frac{f_{0}(z)}{z-w}dz=\lim_{j\rightarrow+\infty}\int_{\partial D_{t_j}}\frac{f_{t_j}(z)}{z-w}dz=\lim_{j\rightarrow+\infty}f_{t_j}(w),$$ 	 
$$\int_{\partial D_{t_j}}f_{0}(z)\phi(z)dz=\lim_{j\rightarrow+\infty}\int_{\partial D_{t_j}}f_{t_j}(z)\phi(z)dz=0$$
for any holomorphic function $\phi$ on a neighborhood of $\overline D_{t_0}$
and 
\begin{equation}
	\label{eq:06299a}\frac{1}{2\pi}\int_{\partial D_{t_0}}|f_{0}|^2\left(\frac{\partial G_D(z,z_0)}{\partial v_z} \right)^{-1}|dz|\le\lim_{j\rightarrow+\infty}\frac{1}{2\pi}\int_{\partial D_{t_j}}|f_{t_j}|^2\left(\frac{\partial G_D(z,z_0)}{\partial v_z} \right)^{-1}|dz|,
\end{equation}
which shows that there exists an element in $H^2_{(c)}(D_{t_0})$ (denoted also by $f_0$) whose nontangential boundary limit is $f_0$ on $\partial D_{t_0}$ by Lemma \ref{l:0-1}. Then, we have
$$f_0(z_0)=\lim_{j\rightarrow+\infty}f_{t_j}(z_0)=1.$$
Equality \eqref{eq:06299b} and \eqref{eq:06299a} show that 
$$\limsup_{t\rightarrow t_0+0}\hat K_{D_t}(z_0)\le \frac{1}{\frac{1}{2\pi}\int_{\partial D_{t_0}}|f_{0}|^2\left(\frac{\partial G_D(z,z_0)}{\partial v_z} \right)^{-1}|dz|}\le\hat K_{D_{t_0}}(z_0).$$
Combining equality \eqref{eq:0619c}, we have $\lim_{t\rightarrow t_0+0}\hat K_{D_t}(z_0)=\hat K_{D_{t_0}}(z_0)$.

Now, we consider the case $t_0\in -2N$.

Note that $G_D(\cdot,z_0)$ is harmonic on $D\backslash\{z_0\}$. Using Riemann mapping Theorem repeatedly,  we obtain a biholomorphic mapping $\mu: \tilde D\rightarrow D_{t_0}$, where $\tilde D$ is a planar region  bounded by finite analytic Jordan curves and $\tilde z_0\in\tilde D$ satisfying $\mu(\tilde z_0)=z_0$. It is clear that $G_{\tilde D}(\cdot,\tilde z_0)=G_D(\cdot,z_0)\circ\mu+\frac{t_0}{2}$.
	Denote that 
	$$\tilde K(t):=\frac{1}{\inf\left\{\|\tilde f\|^2_{H_{(c)}^2(\tilde D)}:\tilde f\in  H^2_{(c)}(\tilde D)\,\&\,\tilde f(\tilde z_0)=1\right\}}.$$
Then we have
\begin{equation}
	\label{eq:0630d}
	\begin{split}
\frac{\pi}{\hat K_{D_{t_0+t}}(z_0)}&=\inf_{\tilde f\in\mathcal{O}(D_{t_0+t})\,\&\,\tilde f( z_0)=1}\liminf_{B\rightarrow0+0}\frac{1}{B}\int_{\{-t-t_0-B<2G_D(\cdot,z_0)<-t-t_0\}}|\tilde f|^2d\lambda_{D}
		\\&=\inf_{\tilde f\in\mathcal{O}(D_{t_0+t})\,\&\,\tilde f( z_0)=1}\liminf_{B\rightarrow0+0}\frac{1}{B}\int_{\{-t-B<2G_D(\cdot,z_0)\circ\mu+t_0<-t\}}|\tilde f\circ\mu|^2|\mu'|^2d\lambda_{\tilde D}\\
		&=\inf_{\tilde f\in\mathcal{O}(\tilde D_t)\,\&\,\tilde f(\tilde z_0)=\mu'(\tilde z_0)}\liminf_{B\rightarrow0+0}\frac{1}{B}\int_{\{-t-B<2G_{\tilde D}(w,\tilde z_0)<-t\}}|\tilde f|^2d\lambda_{\tilde D}\\
		&=|\mu'(\tilde z_0)|^2\frac{\pi}{\tilde K(t)}
	\end{split}
\end{equation}
for any $t\ge0$ by Proposition \ref{l:OSH=H},
where $\tilde D_t:=\{w\in \tilde D:2G_{\tilde D}(w,\tilde z_0)<-t\}=\mu^{-1}(\{z\in  D:2G_D(z,z_0)<-t_0-t\})$, $d\lambda_{\tilde D}$ and $d\lambda_D$ are the Lebesgue measures on $\tilde D$ and $D$ respectively.

As $\tilde D$ is a planar region  bounded by finite analytic Jordan curves, we get that 
$$\lim_{t\rightarrow0+0}\tilde K(t)=\tilde K(0)$$ by above discussions. Combining equality \eqref{eq:0630d}, we obtain that
$$\lim_{t\rightarrow t_0+0}\hat K_{D_t}(z_0)=\hat K_{D_{t_0}}(z_0).$$

Thus, Lemma \ref{l:l-continous} has been proved.
\end{proof}

\begin{Remark}
Following the discussion in the above proof, we know that Lemma \ref{l:uni} holds for $t\in -2N$.
\end{Remark}

It follows from 
Proposition \ref{p:2} and Proposition \ref{l:OSH=H} that 
\begin{equation}
	\label{eq:0709a}\frac{\int_{t_1}^{t_0}\frac{1}{\hat K_{D_t}(z_0)}dt}{e^{-t_1}-e^{-t_0}}\le \frac{e^{t_0}}{\hat K_{D_{t_0}}(z_0)}
\end{equation}
holds for any $0\le t_1<t_0<+\infty$. The following Proposition gives a necessary condition for inequality \eqref{eq:0709a} becomes an equality.

\begin{Proposition}
	\label{p:3}
	 If inequality \eqref{eq:0709a} becomes an equality, then  
	 there exists a holomorphic function $F$ on $\{2G_D(\cdot,z_0)<-t_1\}$ such that $\frac{1}{\hat K_{D_t}(z_0)}=\|F\|_{H^2_{(c)}(D_t)}^2$ for a.e. $t\in(t_1,t_0)$ and  $F(z_0)=1$. 
\end{Proposition}

\begin{proof}Let $f_{t_0}$ be the holomorphic function on $\{2G_D(\cdot,z_0)<-t_0\}$ such that $f_{t_0}(z_0)=1$ and 
	\begin{equation}
		\label{eq:0709b}
		\frac{1}{\hat K_{D_{t
		_0}}(z_0)}=\|f_{t_0}\|_{H^2_{(c)}(D_{t_0})}^2.
	\end{equation}
	Lemma \ref{lem:GZ_sharp2}
shows that for any $B>0$,
there exists a 
holomorphic function $\tilde{F}_{B}$ on $\{2G_D(\cdot,z_0)<-t_1\}$, such that
$\tilde F_{B}(z_0)=1$
and
\begin{equation}
\label{eq:0709c}
\begin{split}
&\int_{\{\psi<-t_1\}}|\tilde{F}_{B}-(1-b_{t_{0},B}(\psi))f_{t_0}|^{2}d\lambda_D\\
\le&\int_{\{\psi<-t_1\}}|\tilde{F}_{B}-(1-b_{t_{0},B}(\psi))f_{t_0}|^{2}e^{-\psi+v_{t_0,B}(\psi)}d\lambda_D
\\\leq&\left(e^{-t_1}-e^{-t_0-B}\right)
\int_{\{\psi<-t_1\}}\frac{1}{B}\mathbb{I}_{\{-t_{0}-B<\psi<-t_{0}\}}|f_{t_0}|^{2}e^{-\psi}d\lambda_D,
\end{split}
\end{equation}
where $\psi:=2G_D(\cdot,z_0)$.
Note that $\liminf_{B\rightarrow0+0}\frac{1}{B}\mathbb{I}_{\{-t_{0}-B<\psi<-t_{0}\}}|f_{t_0}|^{2}e^{-\psi}d\lambda_D<+\infty$ and 
$\liminf_{B\rightarrow0+0}\int_{\{\psi<-t_1\}}|(1-b_{t_{0},B}(\psi))f_{t_0}|^{2}d\lambda_D\le \int_{\{\psi<-t_0\}}|f_{t_0}|^2d\lambda_D<+\infty$, then inequality \eqref{eq:0709c} implies that 
$$\liminf_{B\rightarrow0+0}\int_{\{\psi<-t_1\}}|\tilde F_B|^2d\lambda_D<+\infty.$$
There is a subsequence of $\{\tilde F_B\}_{B>0}$ denoted by $\{\tilde F_{B_j}\}_{j\in\mathbb{Z}_{\ge1}}$ (when $j\rightarrow+\infty$, $B_j\rightarrow0$), which uniformly convergent to a holomorphic function $F$ on $\{\psi<-t_1\}$ on any compact subset of $\{\psi<-t_1\}$. Then we have $F(z_0)=1$, and it follows from inequality \eqref{eq:0709c}, the Fatou's Lemma and Proposition \ref{l:OSH=H} that 
\begin{equation}
	\label{eq:0709d}
	\begin{split}
		&\int_{\{-t_0\le\psi<-t_1\}}|F|^2d\lambda_D+\int_{\{\psi<-t_0\}}|F-f_{t_0}|^2d\lambda_D\\
		\le&\liminf_{j\rightarrow+\infty}\int_{\{\psi<-t_1\}}|\tilde{F}_{B_j}-(1-b_{t_{0},B_j}(\psi))f_{t_0}|^{2}d\lambda_D\\
		\le&\left(e^{-t_1}-e^{-t_0}\right)e^{t_0}
\liminf_{j\rightarrow+\infty}\int_{\{-t_{0}-B_j<\psi<-t_{0}\}}\frac{1}{B_j}|f_{t_0}|^{2}d\lambda_D\\
=&\left(e^{-t_1}-e^{-t_0}\right)e^{t_0}\pi \|f_{t_0}\|_{H^2_{(c)}(D_t)}^2.
	\end{split}
\end{equation}

By Proposition \ref{l:OSH=H}, we have $\frac{1}{\hat K_{D_t}(z_0)}\le \frac{1}{\pi}\liminf_{\epsilon\rightarrow0+0}\frac{1}{\epsilon}\int_{\{-t-\epsilon<\psi<-t\}}| F|^2d\lambda_D$, hence 
$$\int_{t_1}^{t_0}\frac{1}{\hat K_{D_t}(z_0)}dt\le \frac{1}{\pi}\int_{\{-t_0\le\psi<-t_1\}}|F|^2d\lambda_D$$
and the inequality becomes an equality if and only if $$\frac{1}{\hat K_{D_t}(z_0)}= \frac{1}{\pi}\liminf_{\epsilon\rightarrow0+0}\frac{1}{\epsilon}\int_{\{-t-\epsilon<\psi<-t\}}| F|^2d\lambda_D$$ holds a.e. on $(t_1,t_0)$.
As $\frac{\int_{t_1}^{t_0}\frac{1}{\hat K_{D_t}(z_0)}dt}{e^{-t_1}-e^{-t_0}}= \frac{e^{t_0}}{\hat K_{D_{t_0}}(z_0)}$, it follows from equality \eqref{eq:0709b}, inequality \eqref{eq:0709d} and the above characterization that 
$$\frac{1}{\hat K_{D_t}(z_0)}= \frac{1}{\pi}\liminf_{\epsilon\rightarrow0+0}\frac{1}{\epsilon}\int_{\{-t-\epsilon<\psi<-t\}}| F|^2d\lambda_D=\|F\|^2_{H^2_{(c)}(D_t)}$$ holds a.e. on $(t_1,t_0)$.

Thus, Proposition \ref{p:3} holds.	
\end{proof}

\section{Proofs of Theorem \ref{thm1} and Theorem \ref{thm2}}

In this section, we prove Theorem \ref{thm1} and Theorem \ref{thm2}.

\subsection{Proof of Theorem \ref{thm1}}
Firstly, we prove statement $(1)$.

Proposition \ref{p:1} shows that 
$$\frac{\|F_{t_0}\|_{OSH_{\rho}^2(M_{t_0},\psi_{t_0})}^2}{e^{-t_0}c_+(t_0)}\ge\frac{G(t_1)-G(t_0)}{\int_{t_1}^{t_0}c(t)e^{-t}dt}$$ for any $t_0>t_1\ge0$. Following from the concavity of $G(h^{-1}(r))$ and inequality \eqref{eq:0522a2}, we have 
\begin{equation}
	\label{eq:0610b}
	\frac{-G_+'(t_0)}{e^{-t_0}c_+(t_0)}\ge \frac{\|F_{t_0}\|_{OSH_{\rho}^2(M_{t_0},\psi_{t_0})}^2}{e^{-t_0}c_+(t_0)}\ge \frac{G(t_1)-G(t_0)}{\int_{t_1}^{t_0}c(t)e^{-t}dt}\ge \frac{-G_+'(t_1)}{e^{-t_1}c_+(t_1)}, 
\end{equation}
which implies that  $\frac{\|F_{t}\|_{OSH_{\rho}^2(M_{t},\psi_{t})}^2}{e^{-t_0}c_+(t)}$ is increasing on $[0,+\infty)$.

As $G(h^{-1}(r))$ is concave, it is clear that statement $(3)$ can deduce statement $(2)$. Thus, it suffices to prove  statement $(3)$.

Note that $\lim_{t_1\rightarrow t_0-0}c_+(t_1)=c_-(t_0)$, where $c_-(t_0)=\lim_{t\rightarrow t_0-0}c(t)$. Thus, $\frac{-G_+'(t_0)}{e^{-t_0}c_+(t_0)}=\frac{\|F_{t_0}\|_{OSH_{\rho}^2(M_{t_0},\psi_{t_0})}^2}{e^{-t_0}c_+(t_0)}$, i.e. $-G_+'(t_0)=\|F_{t_0}\|_{OSH_{\rho}^2(M_{t_0},\psi_{t_0})}^2$ if the functions $G'_+$ and $c$ are both continuous at $t_0$ according to inequality \eqref{eq:0610b}. As $c(t)e^{-t}$ is decreasing on $(0,+\infty)$ and $G(h^{-1}(r))$ is concave on $[0,\int_0^{+\infty}c(t)e^{-t}]$, we know that functions $G'_+$ and $c$ are both continuous a.e. on $(0,+\infty)$. Thus, Theorem \ref{thm1} holds.

\subsection{Proof of Theorem \ref{thm2}}

In this section, we prove Theorem \ref{thm2}.

\

\emph{Step 1. $(1)\Rightarrow(2)$.}
Denote that 
$$I(t):=\int_{\{\psi<-t\}}|F_0|^2e^{-\varphi}c(-\psi),$$ where $F_0$ is the unique holomorphic $(n,0)$ form on $M$ satisfying that $(F_0-f)\in H^0(Z_0 ,(\mathcal{O} (K_M) \otimes \mathcal{F})|_{Z_0})$ and $G(0)=\int_{ M}|F_0|^2e^{-\varphi}c(-\psi)$. Note that $I(t)$ is decreasing on $[0,+\infty)$ and $\lim_{t\rightarrow+\infty}I(t)=0$, then 
$I'(t)$ exists a.e. on $(0,+\infty)$ and 
\begin{equation}
	\nonumber
	-I'(t)=\lim_{B\rightarrow0+0}\frac{1}{B}\int_{\{-t-B\le\psi<-t\}}|F_0|^2e^{-\varphi}c(-\psi)\ge H(t).
\end{equation}
Thus, combining with statement $(1)$ in Theorem \ref{thm2}, we have
\begin{equation}
	\label{eq:0610c}
	G(0)=I(0)\ge\int_0^{+\infty}-I'(t)dt\ge \int_0^{+\infty}H(t)dt=\int_0^{+\infty}-G_+'(t)dt=G(0),
\end{equation}
where the last ``=" holds by the absolute continuity of $G(t)$.
Inequality \eqref{eq:0610c} implies that $-I'(t)=H(t)$ a.e. on $(0,+\infty)$ and $I(0)=\int_0^{+\infty}-I'(t)dt$, which show that 
\begin{equation}
	\nonumber
	\begin{split}
	\int_{\{\psi<-t\}}|F_0|^2e^{-\varphi}c(-\psi)
	=&I(t)
	=\int_t^{+\infty}-I'(s)ds
	=G(t)
	=\int_{\{\psi<-t\}}|F_t|^2e^{-\varphi}c(-\psi)	
	\end{split}
\end{equation}
for any $t\ge0$. Following from the uniqueness of $F_t$, we have 
$$F_0|_{\{\psi<-t\}}=F_t$$
for any $t\ge0$.
Thus, we have $H(t)=-G'_+(t)=\|F_0\|_{OSH_{\rho}^2(M_{t},\psi_{t})}^2$ for a.e. $t\in(0,+\infty)$.

\

\emph{Step 2. $(2)\Rightarrow(1)$}

\

As $G(t)=\int_{\{\psi<-t\}}|F|^2e^{-\varphi}c(-\psi)$ for any $t\ge0$, then $G_+'(t)=\|F\|_{OSH_{\rho}^2(M_{t},\psi_{t})}^2$ for any $t\in(0,+\infty)$. Combining $H(t)=\|F\|_{OSH_{\rho}^2(M_{t},\psi_{t})}^2$ for a.e. $t\in(0,+\infty)$, we have 
$$H(t)=-G'_+(t)$$
for a.e. $t\in(0,+\infty)$.

Thus, Theorem \ref{thm2} holds.

\section{Proof of Theorem \ref{thm3}}

In this section, we prove Theorem \ref{thm3}.

Following from the monotonicity of $c(t)e^{-t}$, we know that $c_+(t)=c(t)$ a.e. on $(0,+\infty)$.
Note that $OSH_{\rho}^2(M_{t_0},\psi_{t_0})\subset \mathcal{H}^2(M_{t_0},\rho)$ holds for a.e.  $t_0\in(0,+\infty)$ in Theorem \ref{thm3}, then it follows from Proposition \ref{p:2} that inequality 
$$\frac{\int_{t_1}^{t_0}H(t)dt}{\int_{t_1}^{t_0}c(t)e^{-t}dt}\le \frac{e^{t_0}}{c(t_0)}H(t_0)$$
 holds for a.e.  $t_0\in(0,+\infty)$. Thus, Theorem \ref{thm3} holds by the following lemma and $\int_0^{+\infty}H(t)dt\le G(0)<+\infty$.

\begin{Lemma}
	\label{l:concave}Let $a(t)$ and $b(t)$ be nonnegative Lebesgue measurable functions on $(0,+\infty)$, which satisfy that $\int_0^{+\infty}a(t)dt<+\infty$, $\int_0^{+\infty}b(t)dt<+\infty$, $b(t)>0$ a.e. on $(0,+\infty)$ and 
	\begin{equation}
		\label{eq:0615d}\frac{\int_{t_1}^{t_0}a(t)dt}{\int_{t_1}^{t_0}b(t)dt}\le \frac{a(t_0)}{b(t_0)}
	\end{equation}
	holds for a.e. $t_0\in(0,+\infty)$ and any $t_1\in[0,t_0)$.
	 Then $\hat a(\hat b^{-1}(r))$ is concave on $(0,\int_0^{+\infty}b(t)dt]$, where $\hat a(t)=\int_t^{+\infty}a(s)ds$ and $\hat b(t)=\int_t^{+\infty}b(s)ds$ for $t\ge0$. 
\end{Lemma}
\begin{proof}
	As $\hat a'(t)=a(t)$ and $\hat b'(t)=b(t)$ a.e. on $(0,+\infty)$, inequality \eqref{eq:0615d} implies that
	\begin{equation}
		\nonumber
		\frac{\hat a(t_1)-\hat a(t_0)}{\hat b(t_1)-\hat b(t_0)}\le \frac{\hat a'(t_0)}{\hat b'(t_0)}
	\end{equation}
	holds for a.e. $t_0\in(0,+\infty)$ and any $t_1\in[0,t_0)$. Let $r_1=\hat b(t_1)$ and $r_0=\hat b(t_0)$, and denote that $g(r):=\hat a(\hat b(r))$, then we have
	\begin{equation}
		\label{eq:0615e}
		\frac{g(r_1)-g(r_0)}{r_1-r_0}\le g'(r_0)
	\end{equation}
	holds for a.e. $r_0\in(0,\int_0^{+\infty}b(t)dt)$ and any $r_1\in (r_0,\int_0^{+\infty}b(t)dt]$.
	
	We prove the concavity of $g(r)$ by contradiction: if not, there exists $0< r_2<r_3<r_4\le\int_0^{+\infty}b(t)dt$ such that
	\begin{equation}
		\nonumber
		\frac{g(r_3)-g(r_2)}{r_3-r_2}<\frac{g(r_4)-g(r_2)}{r_4-r_2}<\frac{g(r_4)-g(r_3)}{r_4-r_3}.
	\end{equation}
	Consider 
	$$\tilde g(r)=g(r)-g(r_4)-\frac{g(r_4)-g(r_2)}{r_4-r_2}(r-r_4)$$ on $(0,\int_0^{+\infty}b(t)dt]$. As $g(r)$ is continuous  on $(0,\int_0^{+\infty}b(t)dt]$, then $\tilde g(r)$ is continuous  on $(0,\int_0^{+\infty}b(t)dt]$. Note that $\tilde{g}(r_2)=\tilde{g}(r_4)=0$ and $\tilde{g}(r_3)<0$, then it follows from the continuity of $\tilde{g}(r)$ that there exists $r_5\in(r_2,r_4)$ such that 
	$$\tilde{g}(r_5)=\inf_{r\in[r_2,r_4]}\tilde{g}(r)<0.$$ 
	It is clear that there exists  small enough $\epsilon>0$ such that  
	$$\frac{\tilde{g}(r)-\tilde{g}(r_4)}{r-r_4}>0$$
	for any $r\in(r_5-\epsilon,r_5+\epsilon)\subset(r_2,r_4)$, which implies that
	$$ g'(r)\ge  \frac{{g}(r)-{g}(r_4)}{r-r_4}>\frac{g(r_4)-g(r_2)}{r_4-r_2}$$
	for a.e. $r\in(r_5-\epsilon,r_5+\epsilon)$ by inequality \eqref{eq:0615e}. As $g$ is increasing on $(0,\int_0^{+\infty}b(t)dt]$, we have 
	\begin{equation}
		\nonumber
			g(r_5)-g(r_5-\epsilon)\ge\int_{r_5-\epsilon}^{r_5}g'(r)dr
			>\epsilon \frac{g(r_4)-g(r_2)}{r_4-r_2},
	\end{equation}
	i.e., $\tilde g(r_5)>\tilde  g(r_5-\epsilon)$, which contradicts to $\tilde{g}(r_5)=\inf_{r\in[r_2,r_4]}\tilde{g}(r)$.
	
	 Thus, Lemma \ref{l:concave} holds.
	 \end{proof}

\section{Proof of Corollary \ref{c:2}}

In this section, we prove Corollary \ref{c:2} in three steps.

\

\emph{Step 1. $\hat K_{D_t}(z_0)e^{-t}$ is decreasing.}
\

For any $t\ge0,$ note that for any holomorphic $f$ on $D_t$, if there exists $B>0$ such that $\int_{\{-B-t<2G_D(\cdot,z_0)<-t\}}|f|^2d\lambda_D<+\infty$ then $\int_{D_t}|f|^2d\lambda_D<+\infty$.

It follows from Proposition  \ref{l:OSH=H} that 
\begin{equation}
	\nonumber
	\begin{split}
		\frac{1}{\hat K_{D_t}(z_0)}=&\inf\left\{\|\tilde f\|^2_{H^2_{(c)}(D_t)}:\tilde f\in \mathcal{O}(D_t)\,\&\,\tilde f(z_0)=1 \right\}\\
		=&\frac{1}{\pi}\inf\left\{\liminf_{B\rightarrow0+0}\frac{1}{B}\int_{\{-B-t<2G_D(\cdot,z_0)<-t\}}|f|^2d\lambda_D:\tilde f\in \mathcal{O}(D_t)\,\&\,\tilde f(z_0)=1 \right\}.
	\end{split}
\end{equation}
Lemma \ref{l:l-continous} implies that $\frac{1}{\hat K_{D_t}(z_0)}$ is a Lebesgue measurable function with respect to $t\in[0,+\infty)$. Then Theorem \ref{thm3} tell us that $\int_{-\log r}^{+\infty}\frac{1}{\hat K_{D_t}(z_0)}dt$ is concave with respect to $r\in [0,\int_0^{+\infty}c(t)e^{-t}dt]$. By Lemma \ref{l:l-continous}, $\frac{1}{\hat K_{D_t}(z_0)}=\lim_{t_1\rightarrow t+0}\frac{1}{\hat K_{D_{t_1}}(z_0)}$. Thus, we get that $\hat K_{D_t}(z_0)e^{-t}$ is decreasing with respect to $t\in [0,+\infty)$.

\

\emph{Step 2. $(3)\Rightarrow(1)\text{ and } (2)$.}

Assume that $D$ is simply connected, then there is a biholomorphic map $p:\Delta \rightarrow D$ such that $p(o)=z_0$. Following from Lemma \ref{l:partial},  we have 
\begin{equation}
	\nonumber
	\begin{split}
		\frac{\partial G_{\Delta_t}(\cdot,o)}{\partial v_w}&=\frac{\partial (G_{D_t}(\cdot,z_0)\circ p)}{\partial v_w}=|\nabla (G_{D_t}(\cdot,z_0)\circ p)|\\
		&=|\nabla G_{D_t}(\cdot,z_0)\circ p\cdot\nabla p|=\frac{\partial G_{D_t}(\cdot,z_0)}{\partial v_z}\circ p\cdot |p'|,
	\end{split}
\end{equation}
where $\Delta_t:=\{z\in\mathbb{C}:2\log|z|<-t\}$ for any $t\ge0$.
 We have 
\begin{equation}
	\nonumber
	\begin{split}
		&\frac{1}{\hat K_{D_t}(z_0)}\\
		=&\inf\left\{\frac{1}{2\pi}\int_{\partial D_t}|\tilde f(z)|^2\left(\frac{\partial G_{D_t}(\cdot,z_0)}{\partial v_z} \right)^{-1}|dz|:\tilde f(z)\in H^2_{(c)}(D_t)\,\&\, \tilde f(z_0)=1\right\}\\
		=&\inf\left\{\frac{1}{2\pi}\int_{\partial \Delta_t}|\tilde f(w)|^2\left(\frac{\partial G_{\Delta_t}(\cdot,o)}{\partial v_w} \right)^{-1}|p'(w)|^2|dw|:\tilde f(w)\in H^2_{(c)}(\Delta_t)\,\&\, \tilde f(o)=1\right\}\\
		=&\frac{|p'(o)|^2}{\hat K_{\Delta_t}(o)}
	\end{split}
\end{equation}
and
\begin{equation}
	\nonumber
	\begin{split}
		\frac{1}{B_{D_t}(z_0,z_0)}
		=&\inf\left\{\int_{D_t}|\tilde f|^2d\lambda_{D} :\tilde f\in\mathcal{O}(D_t)\,\&\,\tilde f(z_0)=1 \right\}\\
		=&\inf\left\{\int_{D_t}|\tilde f|^2|p'|^2d\lambda_{\Delta} :\tilde f\in\mathcal{O}(\Delta_t)\,\&\,\tilde f(o)=1 \right\}\\
		=&\frac{|p'(o)|^2}{B_{\Delta_t}(o,o)},
	\end{split}
\end{equation}
where $d\lambda_D$ and $d\lambda_{\Delta}$ are the Lebesgue measures on $D$ and $\Delta$ respectively.
It is clear that $\hat K_{\Delta_t}(o)=e^{t}$ and $B_{\Delta_t}(o,o)=\frac{ e^{t}}{\pi}$, thus statements $(1)$ and $(2)$ hold.

\

\emph{Step 3. $(2)\Rightarrow(1)$}
\

Assume that $\hat{K}_{D_t}(z_0)e^{-t}=c$ is a constant function on $[0,+\infty)$. Hence, we have $\frac{\int_{t_1}^{t_0}\frac{\pi}{\hat{K}_{D_t}(z_0)}dt}{e^{-t_1}-e^{-t_0}}=\frac{\pi}{c}=e^{t_0}\frac{\pi}{\hat{K}_{D_{t_0}}(z_0)}$ for any $0\le t_1<t_0<+\infty$. By Proposition \ref{p:3}, there exists a holomorphic function $F$ on $D$ such that $$\frac{1}{\hat{K}_{D_t}(z_0)}=\frac{1}{2\pi}\int_{\partial D_t}|F|^2\left(\frac{\partial G_{D}(\cdot,z_0)}{\partial v_z}\right)^{-1}|dz|$$ 
a.e. on $[0,+\infty)$, and $F(z_0)=1$.

For any $f\in\mathcal{O}(\{2G_D(\cdot,z_0)<-t\})$ satisfying $\int_{\{2G_D(\cdot,z_0)<-t\}}|f|^2<+\infty$ and $\tilde f(z_0)=1$, we have 
$$\frac{1}{2\pi}\int_{\partial D_t}|\tilde f|^2\left(\frac{\partial G_{D}(\cdot,z_0)}{\partial v_z}\right)^{-1}|dz|\ge\frac{1}{\hat{K}_{D_t}(z_0)},$$
which implies that
\begin{equation}
	\nonumber
	\begin{split}
		\int_{D_t}|\tilde f|^2d\lambda_D&=\int_{-\infty}^{-t}\int_{D_s}|\tilde f|^2\left(\frac{\partial 2G_{D}(\cdot,z_0)}{\partial v_z}\right)^{-1}|dz|ds\\
		&\ge \int_{-\infty}^{-t}\frac{\pi}{\hat K_{D_s}(z_0)} ds\\
		&=\int_{-\infty}^{-t}\int_{D_s}|F|^2\left(\frac{\partial 2G_{D}(\cdot,z_0)}{\partial v_z}\right)^{-1}|dz|ds\\
		&=	\int_{D_t}|F|^2d\lambda_D
			\end{split}
\end{equation}
according to Lemma \ref{l:coarea}. Then we have 
\begin{equation}
	\nonumber
	\begin{split}
		&\int_{\{2G_{D}(\cdot,z_0)<-t\}}|F|^2d\lambda_D\\
		=&\inf\left\{\int_{\{2G_D(\cdot,z_0)<-t\}}|\tilde f|^2d\lambda_D:\mathcal{O}(\{2G_D(\cdot,z_0)<-t\})\,\&\,\tilde f(z_0)=1\right\}
	\end{split}
\end{equation}
for any $t\ge0$, which deduces that statement $(1)$ holds by Theorem \ref{thm2}.

\

\emph{Step 4. $(1)\Rightarrow(3)$}
\

Assume that statement $(1)$ holds, then there exists a holomorphic function $F$ on $D$ such that 
\begin{equation}
	\nonumber\begin{split}
		G(t)&:=\inf\left\{\int_{\{2G_D(\cdot,z_0)<-t\}}|\tilde f|^2d\lambda_D:\mathcal{O}(\{2G_D(\cdot,z_0)<-t\})\,\&\,\tilde f(z_0)=1\right\}\\
		&=\int_{\{2G_{D}(\cdot,z_0)<-t\}}|F|^2
	\end{split}
\end{equation}
 for any $t\ge0$ and $F(z_0)=1$ according to Theorem \ref{thm2}.

Let $p:\Delta\rightarrow D$ be the universal covering map, then there is a holomorphic function $f_{z_0}$ on $\Delta$ such that $|f_{z_0}|=p^*\left(e^{G_D(\cdot,z_0)}\right)$. 
Take $t_0$ large enough such that $D_{t_0}$ is simply connected. By the uniqueness property of the minimal function, we have $F=c\frac{dp_*(f_{z_0})}{dz}$ on $D_{t_0}$, where $c$ is a constant and $p_*(f_{z_0})$ is well-defined on $D_{t_0}$. By the uniqueness property of holomorphic form, we have $p^*(Fdz)=cdf_{z_0}$ on $\Delta$. Hence, we get that $p_*(f_{z_0})$ is single-valued on $D$, which shows that $D$ is simply connected (see \cite{suita72}).

Thus, Corollary \ref{c:2} holds.

%%%------------------------------------------------------------------------

\vspace{.1in} {\em Acknowledgements}. 
The authors would like to thank Dr. Shijie Bao and Dr. Zhitong Mi for checking the manuscript and providing some useful suggestions.  The first named author was supported by National Key R\&D Program of China 2021YFA1003100 and NSFC-11825101.

\bibliographystyle{references}
\bibliography{xbib}

\end{document}